\documentclass{amsart}
\usepackage[margin=1 in]{geometry}
\usepackage[T1]{fontenc}							
\usepackage[english]{babel}							
\usepackage[utf8]{inputenc}							
\usepackage{amsmath,amsthm, amsfonts,amsrefs,amssymb, enumerate}
\usepackage{mathrsfs}
\usepackage{mathtools}
\mathtoolsset{mathic} 
\usepackage{tikz}

\usepackage{bm} 
\usepackage{bbm} 
\usepackage{mathrsfs} 
\usepackage{dsfont}
\usepackage{enumitem}
\usepackage{bbm}

\newtheorem{theorem}{Theorem}[section]
\newtheorem{lemma}[theorem]{Lemma}

\theoremstyle{definition}

\newtheorem{remark}[theorem]{Remark} 

\definecolor{pink}{rgb}{1,0,1}
\def\beq{\begin{eqnarray}}
\def\eeq{\end{eqnarray}}

\newcommand*{\R}{\mathbb{R}}
\newcommand*{\N}{\mathbb{N}}
\newcommand*{\C}{\mathbb{C}}

\newcommand*{\supp}{\text{supp}}

\usepackage{graphicx}								
\usepackage{subfig}									
\usepackage{listings}								
\usepackage{eso-pic}								

\usepackage{float} 									

\usepackage[labelfont=bf, textfont=normal,
			justification=justified,
			singlelinecheck=false]{caption}

\usepackage{hyperref}								
\hypersetup{colorlinks, citecolor=black,
   		 	filecolor=black, linkcolor=black,
    		urlcolor=black}

\usepackage[textsize=tiny]{todonotes}   

\usepackage{longtable}

\title[Exponential Bounds in the Neumann Heat Trace of Polygons]{
Beyond Three Terms: Exponential Bounds in the Neumann Heat Trace of Polygons}
\author[G.\@ M\aa rdby]{Gustav M\aa rdby}
\address{Mathematical Sciences, Chalmers University of Technology and University
of Gothenburg, 412 96 Gothenburg, Sweden}
\email{mardby@chalmers.se}

\begin{document}
\maketitle
\begin{abstract}
We study the short-time asymptotic behavior of the heat trace associated with the Neumann Laplacian on polygonal domains in $\R^2$. By establishing locality principles for the heat kernel near corners, edges, and the interior, we approximate the heat kernel on the polygon by model heat kernels defined on infinite sectors, half-planes, and the full plane, respectively. Although it is known that the Neumann heat trace of polygons admits a three-term asymptotic expansion followed by an exponentially small remainder, an explicit estimate for the exponent in this remainder term is not known. In this article, we provide such an estimate. 
We also discuss whether the exponent is sharp, and how it relates to known results. Finally, we discuss issues that arise when trying to extend the results to Robin boundary conditions.
\end{abstract}

\section{Introduction} \label{s:intro}
Let $\Omega \subset \R^2$ be a bounded domain in the Euclidean plane. 
The classical Laplace eigenvalue equation is to find all functions $u : \Omega \to \C$ and $\lambda \in \C$ such that
\begin{equation} \label{eq:eigenvalue}
    \Delta u = \lambda u \text{ in } \Omega.
\end{equation}
Here, we use the following sign convention for the Laplace operator:
\begin{equation*}
    \Delta = -\frac{\partial^2}{\partial x_1^2} -\frac{\partial^2}{\partial x_2^2}.
\end{equation*}
The eigenvalue equation plays a central role in both mathematics and physics. This should not be too surprising, since the Laplace operator appears in several fundamental evolution equations, such as the heat, wave, and Schrödinger equations. To obtain a well-posed spectral problem, one equips \eqref{eq:eigenvalue} with a boundary condition. The most classical choice is probably the \textit{Dirichlet boundary condition},
\begin{equation*}
    u|_{\partial\Omega} = 0,
\end{equation*}
which, depending on the underlying physical model, represents, for example, a membrane fixed along its boundary in the case of the wave equation, or a boundary held at zero temperature in the case of the heat equation. The solutions $u = u_k$ of \eqref{eq:eigenvalue} that satisfy the Dirichlet boundary condition and are not identically zero are called \textit{eigenfunctions}, and the corresponding $\lambda = \lambda_k$ are the \textit{eigenvalues}. The collection of all eigenvalues, counted with multiplicity, forms the \textit{spectrum} of the Laplacian on $\Omega$. The spectrum plays a fundamental role in physics as it, among other things, describes the resonant frequencies of vibrating membranes, the rate at which heat diffuses through a medium, and the energy levels of quantum systems \cite{kac1966can}. Any quantity that can be recovered from the spectrum is called a \textit{spectral invariant}. When $\Omega$ is bounded and the Laplacian is equipped with the Dirichlet boundary condition, the spectrum is discrete and positive, forming a sequence
\begin{equation*}
    0 < \lambda_1 \leq \lambda_2 \leq \lambda_3 \leq \dots \uparrow \infty.
\end{equation*}
Moreover, the associated normalized eigenfunctions $\{u_k\}_{k=1}^\infty$ form an orthonormal basis of $L^2(\Omega)$ \cite[Thm. 2.1.20]{levitin2023topics}. 

It is natural to ask to what extent the spectrum determines the geometry of the domain. In other words, what geometric features are and are not spectral invariants? Is perhaps all of the geometry a spectral invariant? In the planar case, this idea was popularized by Mark Kac’s celebrated question from 1966, \textit{``Can one hear the shape of a drum?''} \cite{kac1966can}. His question initiated a massive amount of research on the subject; see \cite{maardby2026century} for an overview. Although the answer to Kac's original question has been known for over 30 years \cite{gordon1992isospectral}, many related questions remain open.

One powerful tool for obtaining spectral invariants is the so-called \textit{heat trace}. To see how it arises, consider the heat equation with the Dirichlet boundary condition:
\begin{equation} \label{eq:heat_equation}
    \begin{cases}
        (\frac{\partial}{\partial t} + \Delta)u(t,x) = 0 \text{ in } (0,\infty) \times \Omega, \\
        u(t,x) = 0 \text{ on } (0,\infty) \times \partial \Omega, \\
        u(0,x) = f(x) \in L^2(\Omega) \text{ in } \Omega.
    \end{cases}
\end{equation}
The fundamental solution of \eqref{eq:heat_equation} is known as the \textit{heat kernel} $H_\Omega : (0,\infty) \times \Omega \times \Omega \to \R$. This means that, when $\Delta$ acts in either $x$ or $y$, $H_\Omega(t,x,y)$ satisfies
\begin{align*}
    \bigg(\frac{\partial}{\partial t} + \Delta\bigg) H_\Omega(t,x,y) = 0&, \\
    H_\Omega(t,x,y) = H_\Omega(t,y,x)&, \\
    \lim_{t \downarrow 0} H_\Omega(t,x,y) = \delta(x-y)& \text{ in the sense of distributions},
\end{align*}
and the heat equation \eqref{eq:heat_equation} has the solution
\begin{equation*}
    u(t,x) = \int_\Omega H_\Omega(t,x,y) f(y) dy.
\end{equation*}
Moreover, if $\{\lambda_k\}_{k=1}^\infty$ and $\{u_k\}_{k=1}^\infty$ are the Dirichlet eigenvalues and normalized eigenfunctions, respectively, then we also have (see \cite[Thm. 6.1.2]{levitin2023topics})
\begin{equation*}
    H_\Omega(t,x,y) = \sum_{k=1}^\infty e^{-\lambda_k t} \, u_k(x)u_k(y), \,\,x,y \in \Omega, \,\, t > 0.
\end{equation*}
This series converges absolutely on $(0,\infty) \times \Omega \times \Omega$ and extends smoothly to $t = 0$ off the diagonal $x = y$.
Integrating the heat kernel along the diagonal yields the spectral invariant known as the \textit{heat trace}:
\begin{equation*}
    Z(t) := \int_\Omega H_\Omega(t,x,x)\,dx = \sum_{k=1}^\infty e^{-\lambda_k t}, \,\, t > 0.
\end{equation*}
Although this series diverges as $t \downarrow 0$, it converges uniformly for $t \geq T$ for any $T > 0$. This can be shown using \textit{Weyl's law} \cite{weyl1912asymptotische}, which says that the Dirichlet eigenvalues for any bounded domain $\Omega$ in the plane have the asymptotic growth rate
\begin{equation} \label{eq:Weyls_law}
    \lambda_k \sim \frac{4\pi k}{|\Omega|}, \,\, k \to \infty,
\end{equation}
where $|\Omega|$ denotes the area of $\Omega$. For rectangles, \eqref{eq:Weyls_law} follows directly from the explicit eigenvalue expressions. More generally, if $\Omega$ has piecewise smooth boundary, Weyl's law can be proved by approximating $\Omega$ from inside and outside by finite unions of squares and applying domain monotonicity (see e.g. \cite[Thm. 2.6]{maardby2026century}). For arbitrary bounded domains, Weyl's law was established by Rozenblum \cite{rozenbljum1972eigenvalues}.

Using \eqref{eq:Weyls_law}, one can show that (see \cite[pages 29-30]{maardby2026century})
\begin{equation*}
    Z(t) \sim \frac{|\Omega|}{4\pi t}, \,\, t \downarrow 0.
\end{equation*}
Building on Weyl’s law, Pleijel \cite{pleijel1954study} refined the heat trace asymptotics using Green’s functions and methods due to Carleman to show that
\begin{equation*}
    Z(t) \sim \frac{|\Omega|}{4\pi t} - \frac{|\partial\Omega|}{8\sqrt{\pi t}}, \,\, t \downarrow 0,
\end{equation*}
for any smoothly bounded planar domain, where $|\partial\Omega|$ denotes the perimeter of $\Omega$. In 1967, McKean and Singer \cite{mckean1967curvature} extended this further in the setting of compact Riemannian manifolds by obtaining a third, topological term in the expansion. By relating the heat trace coefficients to the curvature tensor and carefully manipulating certain Levi sums, they showed that 
\begin{equation} \label{eq:heat_trace_expansion_three_terms}
    Z(t) \sim \frac{|\Omega|}{4\pi t} - \frac{|\partial\Omega|}{8\sqrt{\pi t}} + \frac{\chi(\Omega)}{6}, \,\, t \downarrow 0.
\end{equation}
This implies that, in addition to the area and perimeter, the Euler characteristic, or more intuitively the number of holes, is a spectral invariant for smooth planar domains. It is now known that the heat trace of any smoothly bounded domain with the Dirichlet boundary condition has an asymptotic expansion as $t \downarrow 0$ of the form
\begin{equation*}
    Z(t) = \frac{a_{-1}}{t} + \frac{a_{-1/2}}{\sqrt{t}} + a_0 + a_{1/2}\sqrt{t} + a_1t + \dots,
\end{equation*}
where each coefficient $a_{j/2}$, can be expressed as the sum of two integrals: one over $\Omega$ involving a universal polynomial in the Gauss curvature $K$ of the metric and its covariant derivatives, and one over $\partial\Omega$ involving another universal polynomial in the boundary curvature $\kappa$ and its derivatives. However, explicit formulas for higher-order coefficients become increasingly complicated. See \cite{watanabe2000plane} for further terms.

The three-term asymptotic expansion \eqref{eq:heat_trace_expansion_three_terms} had already been conjectured by Kac \cite{kac1966can} one year before McKean and Singer. By approximating smooth domains with polygons, Kac heuristically obtained the same expansion, though his arguments were not rigorous. Nevertheless, this motivates the study of heat trace asymptotics of polygonal domains. To the best of our knowledge, this was first studied by Brownell in 1957 \cite{brownell1957improved}, who conjectured that polygons only have three heat trace terms followed by exponential decay. A complicated proof, with several missing details, of this is given in 
\cite{bailey1961removal1}. A complete and detailed derivation was eventually given by van den Berg and Srisatkunarajah in 1988 \cite{van1988heat}, who showed that for a polygon $\Omega$, 
\begin{equation} \label{eq:Dirichlet_final_bound}
    Z(t) = \frac{|\Omega|}{4\pi t} - \frac{|\partial\Omega|}{8\sqrt{\pi t}} + \sum_{i=1}^n \frac{\pi^2 - \gamma_i^2}{24\pi\gamma_i} + \mathcal{O}(e^{-R^2\sin^2(\gamma/2)/(16t)}), \,\, t \downarrow 0,
\end{equation}
where $\gamma_1,\dots,\gamma_n$ are the interior angles, and $\gamma$ and $R$ are defined in \eqref{eq:def_gamma} and \eqref{eq:def_R} below. Thus, in contrast to smooth domains, which typically have infinitely many non-zero heat trace terms, polygons have only three, followed by exponential decay. Moreover, the third heat trace term changes from $\frac{\chi(\Omega)}{6}$, which equals $\frac{1}{6}$ for simply connected domains, to
\begin{equation*}
    \sum_{i=1}^n \frac{\pi^2 - \gamma_i^2}{24\pi\gamma_i},
\end{equation*}
which can be shown to always be strictly greater than $\frac{1}{6}$ (see e.g. \cite[Thm. 6.4]{maardby2025spectral}). Van den Berg and Srisatkunarajah's proof of \eqref{eq:Dirichlet_final_bound} relies on expressing the Green's function for an infinite sector as a Kontorovich-Lebedev transform, in combination with giving a probabilistic representation of the heat kernel in terms of Brownian motion, a technique well suited for the Dirichlet boundary condition. 

In this article, we study the Laplace operator on polygons subject to the \textit{Neumann boundary condition}, arguably the second most natural boundary condition after Dirichlet. This requires the normal derivative of the eigenfunctions to vanish at the boundary,
\begin{equation*}
    \frac{\partial u}{\partial n}\bigg|_{\partial\Omega} = 0.
\end{equation*}
In particular, this boundary condition models an insulated boundary in the heat equation, where no heat flows across the boundary. However, in the case of polygons, the boundary is not everywhere smooth, and the normal is not defined at the corners. This is not a big problem, since the set of corner points is small. 
More precisely, the issue can be rigorously resolved by instead interpreting the Neumann problem in the weak sense. Namely, one seeks all $u \in H^1(\Omega)$, not identically zero, and $\lambda \in \mathbb{C}$ such that
\begin{equation*}
    \langle \nabla u, \nabla v\rangle_{L^2(\Omega)} = \lambda \langle u, v \rangle_{L^2(\Omega)} \text{ for all } v \in H^1(\Omega).
\end{equation*}
By the weak spectral theorem for the Neumann Laplacian, there is an orthonormal basis for $L^2(\Omega)$ consisting of weak eigenfunctions, and the corresponding eigenvalues form a discrete non-negative increasing sequence tending to infinity. Furthermore, elliptic regularity implies that the weak eigenfunctions are in fact smooth in $\Omega$, and smooth up to the smooth parts of the boundary, where they satisfy the Neumann condition in the classical sense. These results hold more generally for any bounded domain with Lipschitz boundary. See e.g. \cite[Ch. 2]{levitin2023topics} for details.

Thus, just as in the Dirichlet case, the Neumann Laplacian on a polygon has discrete spectrum, this time starting with zero, 
\begin{equation*}
    0 = \lambda_0 < \lambda_1 \leq \lambda_2 \leq \dots \uparrow \infty,
\end{equation*}
and a well defined heat trace
\begin{equation*}
    Z(t) = \sum_{k=0}^\infty e^{-\lambda_k t}, \,\, t > 0.
\end{equation*}
It is therefore natural to ask: What does the short-time asymptotic expansion of the heat trace look like in this setting? Kokotov \cite[Thm. 1]{kokotov2010spectral} and Fursaev~\cite{fursaev1994heat} indirectly showed that the heat trace of the Neumann Laplacian on a polygon satisfies 
\begin{equation} \label{eq:heat_trace_expansion_polygon_Neumann}
    Z(t) = \frac{|\Omega|}{4\pi t} + \frac{|\partial\Omega|}{8\sqrt{\pi t}} + \sum_{i=1}^n \frac{\pi^2 - \gamma_i^2}{24\pi\gamma_i} + \mathcal{O}(e^{-\alpha/t}), \,\, t \downarrow 0,
\end{equation}
for some $\alpha > 0$. In other words, the first three terms are the same as for the Dirichlet boundary condition, except that the second term now has a plus sign instead of a minus sign, followed by an exponentially small remainder. However, they were unable to compute $\alpha$, or even provide a lower bound. This is because they did not treat planar polygons directly. Instead, they considered compact polyhedral surfaces, whose exponential remainder term was obtained by integrating local heat kernels associated with different regions of the surface. These local heat kernels were not given in explicit form, but only characterized through their general asymptotic behavior and known exponential decay properties. The heat trace of a polyhedral surface decomposes as the sum of the Dirichlet and Neumann heat traces of the associated polygons, from which \eqref{eq:heat_trace_expansion_polygon_Neumann} can be deduced. 
However, since the decay rates of these local contributions were not computed explicitly, no estimate of the exponent $\alpha$ can be extracted from their argument. 
In the Dirichlet case, \cite{van1988heat} obtain an estimate of the exponent using a probabilistic representation of the heat kernel in terms of Brownian motion. In the Neumann case, a corresponding representation is more involved, and the argument in \cite{van1988heat} does not transfer directly in a straightforward way. Instead of pursuing this probabilistic route, we develop an alternative approach based on \textit{locality principles}, which allows us to obtain an explicit remainder. To state our main result, we start as in \cite{van1988heat} by defining
\begin{equation} \label{eq:def_gamma}
    \gamma := \min_i \gamma_i,
\end{equation}
\begin{equation} \label{eq:def_B_i}
    B_i(r) := \{x \in S_{\gamma_i} : d(x,P_i) < r\}, \,\, i = 1,\dots,n, \,\, r > 0,
\end{equation}
\begin{equation} \label{eq:def_R}
    R := \frac{1}{2} \sup\bigg\{r > 0 : B_i(r) \cap B_j(r) = \emptyset \text{ for all } i \neq j \text{ and } \bigcup_{k=1}^n B_k(r) \subset \Omega\bigg\}.
\end{equation}
Here, $P_i$ are the corners of the polygon, and $S_{\gamma_i}$ is the infinite sector with vertex $P_i$ and interior angle $\gamma_i$ such that $\partial S_{\gamma_i}$ contains the two edges of $\Omega$ that meet at $P_i$. Then we have the following.
\begin{theorem} \label{thm:main_result}
    Let $\Omega \subset \R^2$ be an $n$-sided polygon with interior angles $\gamma_1,\dots,\gamma_n$. Let $\gamma$ and $R$ be as in \eqref{eq:def_gamma} and \eqref{eq:def_R}, respectively. Then, there is a $C > 0$ such that the Neumann heat trace satisfies
    \begin{equation} \label{eq:Neumann_final_bound}
    \bigg|Z(t) - \frac{|\Omega|}{4\pi t} - \frac{|\partial\Omega|}{8\sqrt{\pi t}} - \sum_{i=1}^n \frac{\pi^2-\gamma_i^2}{24\pi\gamma_i}\bigg| \leq Ce^{-R^2\sin^2(\gamma/2)/(16t)}, \,\, t > 0.
    \end{equation}
\end{theorem}
The proof of Theorem \ref{thm:main_result} relies on locality principles, specifically Lemmas \ref{lemma:locality_principle_general}-\ref{lemma:locality_principle_interior}, which approximate the Neumann heat kernel $H_\Omega(t,x,y)$ by those on model domains, depending on the positions of $x,y \in \Omega$. Away from the boundary, $H_\Omega(t,x,y)$ is approximated by the heat kernel $H_{\R^2}(t,x,y)$ on the Euclidean plane:
\begin{equation} \label{eq:heat_kernel_R^2}
    H_{\R^2}(t,x,y) = \frac{1}{4\pi t} e^{-|x-y|^2/(4t)}.
\end{equation}
Near an edge but away from the corners, $H_\Omega(t,x,y)$ is approximated by the Neumann heat kernel $H_{\R^2_+}(t,x,y)$ on the Euclidean half-plane $\R^2_+ = \{x = (x_1,x_2) \in \R^2 : x_2 > 0\}$,
\begin{equation} \label{eq:heat_kernel_R^2_+}
    H_{\R^2_+}(t,x,y) = H_{\R^2}(t,x,y) + H_{\R^2}(t,x,y^*),
\end{equation}
where $y^* = (y_1,-y_2)$ is the reflection of $y = (y_1,y_2)$ across $\partial\R^2_+ = \{(x_1,x_2) \in \R^2 : x_2 = 0\}$. Near a corner $P_i$ with angle $\gamma_i$, $H_\Omega(t,x,y)$ is approximated by the Neumann heat kernel $H_{\gamma_i}(t,x,y)$ on an infinite sector $S_{\gamma_i}$ with interior angle $\gamma_i$. The heat kernel $H_{\gamma_i}(t,x,y)$ is more involved than $H_{\R^2}(t,x,y)$ and $H_{\R^2_+}(t,x,y)$, but well studied; see e.g. \cite{carslaw1910green}, \cite{fedosov1963asymptotic, fedosov1964asymptotic}, \cite{mckean1967curvature}, \cite{cheeger1983spectral}.
Using the locality principles, we decompose the heat trace into corner, edge, and interior contributions, and estimate the difference between the true heat kernel and the appropriate model heat kernel on each region.

Using Theorem \ref{thm:main_result} together with the Dirichlet heat trace expansion \eqref{eq:Dirichlet_final_bound}, we can also obtain a corresponding short-time asymptotic expansion for the heat trace of compact polyhedral surfaces. As explained in \cite[p. 3766]{hezari2017neumann}, from any polygon $\Omega$ we may form a compact Euclidean surface with conical singularities $\Sigma$ by gluing together two copies of $\Omega$ along their boundaries. The resulting surface has vertices $P_1,\dots,P_n$ with angles twice the size of the interior angles of $\Omega$, and the Friedrichs extension of the Laplacian on $C_0^\infty (\Sigma \backslash\{P_1,\dots,P_n\})$ has spectrum equal to the union of the Dirichlet and Neumann spectra of $\Omega$, counting multiplicity. In particular, the heat trace $Z(t)$ of $\Sigma$ satisfies
\begin{equation*}
    Z(t) = Z^D(t) + Z^N(t),
\end{equation*}
where $Z^D(t)$ and $Z^N(t)$ are the Dirichlet and Neumann heat traces of $\Omega$, respectively. Combining this with \eqref{eq:Dirichlet_final_bound}, Theorem \ref{thm:main_result}, and the observation
\begin{equation*}
    \bigg(\frac{|\Omega|}{4\pi t} - \frac{|\partial\Omega|}{8\sqrt{\pi t}} + \sum_{i=1}^n \frac{\pi^2-\gamma_i^2}{24\pi\gamma_i}\bigg) + \bigg(\frac{|\Omega|}{4\pi t} + \frac{|\partial\Omega|}{8\sqrt{\pi t}} + \sum_{i=1}^n \frac{\pi^2-\gamma_i^2}{24\pi\gamma_i}\bigg) = \frac{|\Omega|}{2\pi t} + \sum_{i=1}^n \frac{\pi^2-\gamma_i^2}{12\pi\gamma_i},
\end{equation*}
we obtain the following result.
\begin{theorem}
Let $\Omega \subset \R^2$ be an $n$-sided polygon with vertices $P_1,\dots,P_n$ of interior angles $\gamma_1,\dots,\gamma_n$. Let $\gamma$ and $R$ be as in \eqref{eq:def_gamma} and \eqref{eq:def_R}, respectively. Let $\Sigma$ be the corresponding compact polyhedral surface, and let $\Delta$ denote the Friedrichs extension of the Laplace operator defined on $C_0^\infty(\Sigma \setminus \{P_1,\dots,P_n\})$. Then, there is a $C > 0$ such that the heat trace $Z(t)$ of $\Sigma$ satisfies
\begin{equation*} 
\bigg|Z(t) - \frac{|\Omega|}{2\pi t} - \sum_{i=1}^n \frac{\pi^2-\gamma_i^2}{12\pi\gamma_i}\bigg| \leq Ce^{-R^2\sin^2(\gamma/2)/(16t)}, \,\, t > 0.
\end{equation*}
Equivalently, if we denote the conical angles of $\Sigma$ by $\beta_1,\dots,\beta_n$, then
\begin{equation*}
\bigg|Z(t) - \frac{|\Sigma|}{4\pi t} - \sum_{i=1}^n \frac{4\pi^2-\beta_i^2}{24\pi\beta_i}\bigg| \leq Ce^{-R^2\sin^2(\gamma/2)/(16t)}, \,\, t > 0.
\end{equation*}
\end{theorem}
The article is organized as follows. In \S\ref{sec:locality_principles}, we introduce the locality principles and show how to obtain them for the different parts of the polygon. In \S\ref{sec:heat_trace}, we use the locality principles to prove Theorem \ref{thm:main_result}. Finally, in \S\ref{sec:outlook}, we offer some possible directions for further study.


\section*{Acknowledgments} 
The author would like to thank Julie Rowlett and Simon Larson for invaluable support during this project. The author would also like thank Magnus Goffeng, Grigori Rozenblum, Jeffrey Steif, and Elton Hsu for helpful discussions.

\section{Locality principles} \label{sec:locality_principles}
To estimate the Neumann heat trace of a polygonal domain $\Omega$, we rely on the principle that the heat kernel $H_\Omega(t,x,y)$ can be approximated locally by the heat kernel on simpler model domains. To formalize this idea, we use the notion of exact geometric matches, which was introduced in \cite[]{nursultanov2019hear}. Let $\Omega_0, \Omega,S \subset \mathbb{R}^n$ be domains with $\Omega_0 \subset \Omega \cap S$. We say that $\Omega$ and $S$ are \textit{exact geometric matches} on $\Omega_0$ if $\partial \Omega_0 \cap \partial \Omega = \partial \Omega_0 \cap \partial S$. This setting allows us to compare the heat kernels $H_\Omega(t,x,y)$ and $H_S(t,x,y)$ for points $x, y \in \Omega_0$ as $t \downarrow 0$.  

\subsection{General patchwork construction} \label{sec:patchwork_construction}
Let $\Omega_0, \Omega, S \subset \R^n$ be domains such that $\Omega$ and $S$ are exact geometric matches on $\Omega_0$, and assume that the associated Laplace operators are equipped with the same boundary condition. In \cite[Thm. 4, Thm. 5]{nursultanov2019hear}, it is shown that for the Neumann and Robin boundary conditions, one has
\begin{equation} \label{eq:weak_bound}
|H_\Omega(t,x,y) - H_S(t,x,y)| = \mathcal{O}(t^\infty), \,\, x,y \in \Omega_0, \,\, t \downarrow 0,
\end{equation}
where the notation $\mathcal{O}(t^\infty)$ indicates that the error decays faster than any power of $t$ as $t \downarrow 0$. The corresponding statement for the Dirichlet boundary condition had been established earlier and is known as \textit{Kac's principle of not feeling the boundary}. Note that \cite{nursultanov2019hear} has certain assumptions about $\Omega$ and $S$, such as the $(\epsilon,h)$-cone condition and that the second fundamental forms of $\partial\Omega$ and $\partial S$ are bounded below. However, such assumptions are satisfied in our setting where $\Omega$ is a polygon and $S$ is either an infinite sector, the Euclidean half-plane, or the Euclidean plane; see \cite[Remark 1]{nursultanov2019hear}.

Although \eqref{eq:weak_bound} is sufficient for the purposes of \cite{nursultanov2019hear}, it does not provide explicit control on the decay rate. For our application, such quantitative information is needed in order to obtain an explicit remainder estimate for the heat trace. We therefore revisit the proof of \cite[Thm. 4]{nursultanov2019hear} in detail, but make the decay rate explicit.
Throughout, by an \textit{$\epsilon$-neighborhood} of a subset $A \subset \Omega$ we mean 
\begin{equation*}
    \{x \in \Omega : d(x,A) < \epsilon\}.
\end{equation*}
In particular, neighborhoods are taken relative to $\Omega$ rather than to $\R^2$. 

\begin{lemma}[Locality principle] \label{lemma:locality_principle_general}
    Let $\Omega_0,\Omega,S \subset \R^2$ be domains with Lipschitz boundary. 
    Assume there are constants $C, c, m > 0$ such that the Neumann heat kernels $H_{D_j}(t,x,y)$ satisfy the uniform bounds 
    \begin{align} \label{eq:H_S_bounds}
        \begin{split}
            |H_{D_j}(t,x,y)| \leq \frac{C}{t^{m}}e^{-|x-y|^2/(c t)}, \,\, x,y \in D_j \cap \Omega, \,\, 0 < t < 1, \\
            |\nabla H_{D_j}(t,x,y)| \leq \frac{C}{t^{m}}e^{-|x-y|^2/(c t)}, \,\, x,y \in D_j \cap \Omega, \,\, 0 < t < 1,
        \end{split}
    \end{align}
    for $j = 1,2$, where $D_1 = \Omega$ and $D_2 = S$. Then, for $\epsilon,\delta > 0$ such that $\Omega$ and $S$ are exact geometric matches on an $(\epsilon+\delta)$-neighborhood of $\Omega_0$, there is a $C' > 0$ such that 
    \begin{equation} \label{eq:locality_principle_general}
        |H_\Omega(t,x,x) - H_{S}(t,x,x)| \leq C'e^{-\epsilon^2/(ct)}, \,\, x \in \Omega_0, \,\, 0 < t < 1.
    \end{equation}
\end{lemma}
\begin{proof}
Let $\{\chi_1, \chi_2\}$ be a smooth partition of unity for $\Omega$ such that 
\begin{align*}
    \chi_1 = 1 \text{ in } \Omega_0, \text{ and } \chi_1 < 1 \text{ elsewhere}, \\
    \chi_1 > 0 \text{ in an } \epsilon\text{-neighborhood of } \Omega_0, \text{ and } \chi_1 = 0 \text{ elsewhere}.
\end{align*}
Write $\epsilon' := \delta/3$ and let $0 \leq \Tilde{\chi}_1, \Tilde{\chi}_2 \leq 1$ be smooth functions on $\Omega$ such that for $j = 1,2$,
\begin{align*}
    &\Tilde{\chi}_j = 1 \text{ in an } 2\epsilon' \text{-neighborhood of } \supp(\chi_j), \text{ and } \tilde{\chi}_j < 1 \text{ elsewhere}, \\
    &\tilde{\chi}_j > 0 \text{ in an } \epsilon' \text{-neighborhood of } \{\tilde{\chi}_j = 1\}, \text{ and } \tilde{\chi}_j = 0 \text{ elsewhere}.
\end{align*}
In particular, $\Omega$ and $S$ are exact geometric matches on $\supp(\tilde{\chi}_1)$. Now define the patchwork heat kernel 
\begin{equation} \label{eq:H_Omega-G}
    G(t,x,y) := \Tilde{\chi}_1(x)H_{S}(t,x,y)\chi_1(y) + \Tilde{\chi}_2(x)H_{\Omega}(t,x,y)\chi_2(y), \,\, x,y \in \Omega, \,\, t > 0.
\end{equation}
Since $\Tilde{\chi}_1(x)H_{S}(t,x,y)\chi_1(y) = 0$ when either $x$ or $y$ is outside of $\supp(\tilde{\chi}_1)$, $G(t,x,y)$ is well defined and smooth on all of $\Omega$, even outside of $S$. Moreover, for $x,y \in \Omega_0$ we have $\Tilde{\chi}_1(x) = \chi_1(y) = 1$ and $\chi_2(y) = 0$, hence 
\begin{equation} \label{eq:G=H_S}
    G(t,x,y) = H_S(t,x,y).
\end{equation}
It is therefore enough to estimate $|H_\Omega(t,x,y) - G(t,x,y)|$ for $x,y \in \Omega_0$. To achieve this, define
\begin{align*}
    E(t,x,y) := \left(\frac{\partial}{\partial t} + \Delta\right)G(t,x,y), \,\, x,y \in \Omega, \,\, t > 0,
\end{align*}
where $\Delta$ acts in $x$. From
\begin{align*}
    \frac{\partial H_{D_j}}{\partial t} + \Delta H_{D_j} = 0 \text{ in } (0,\infty) \times {D_j} \times {D_j}, \\
    \Delta(fg) = \Delta(f)g + f\Delta(g) + 2\nabla f\cdot\nabla g, 
\end{align*}
it is straightforward to verify that
\begin{equation*} 
    E(t,x,y) = [\Delta, \Tilde{\chi}_1(x)]H_{S}(t,x,y)\chi_1(y) + [\Delta, \Tilde{\chi}_2(x)]H_{\Omega}(t,x,y)\chi_2(y).
\end{equation*}
Then, since $[\Delta, \Tilde{\chi}_j(x)] = 0$ in $\supp(\chi_j)$, it follows that
\begin{equation} \label{eq:E_simplification}
    E(t,x,y) = 
    \begin{cases}
        [\Delta, \Tilde{\chi}_1(x)]H_{S}(t,x,y)\chi_1(y), \text{ if } \chi_1(x) = 0, \\
        [\Delta, \Tilde{\chi}_2(x)]H_{\Omega}(t,x,y)\chi_2(y), \text{ if } \chi_1(x) = 1, \\
        0, \text{ otherwise}.
    \end{cases}
\end{equation}
Moreover, for $|x-y| < 2\epsilon'$, we have either $\chi_j(y) = 0$ or $\tilde{\chi}_j(x) = 1$ for $j = 1,2$, hence $E(t,x,y) = 0$. In other words, $E(t,x,y) \neq 0$ only if $|x-y| \geq 2\epsilon'$. Then, since
\begin{equation*}
    [\Delta, \Tilde{\chi}_j(x)]H_{D_j}(t,x,y)\chi_j(y) = (\Delta(\tilde{\chi}_j(x))H_{D_j}(t,x,y)\chi_j(y) + 2 \nabla \tilde{\chi}_j(x) \cdot \nabla(H_{D_j}(t,x,y))\chi_j(y),
\end{equation*}
it follows from \eqref{eq:H_S_bounds} that
\begin{equation*}
    |[\Delta, \Tilde{\chi}_j(x)]H_{D_j}(t,x,y)\chi_j(y)| \leq \frac{C}{t^{m}}  e^{-(2\epsilon')^2/(c t)} = \frac{C}{t^{m}} e^{-3(\epsilon')^2/(c t)}e^{-(\epsilon')^2/(c t)}, \,\, x,y \in \Omega, \,\, 0 < t < 1,
\end{equation*}
for some $C > 0$. Since 
\begin{equation*}
    \frac{C}{t^{m}} e^{-3(\epsilon')^2/(c t)}
\end{equation*}
is bounded for $0 < t < 1$, it follows that there is a $C > 0$ such that
\begin{equation} \label{eq:E_fresh_bound}
    |E(t,x,y)| \leq Ce^{-(\epsilon')^2/(ct)}, \,\, x,y \in \Omega, \,\, 0 < t < 1.
\end{equation}
Next, define the convolution
\begin{equation} \label{eq:E*E_def}
    E * E (t,x,y) := \int_0^t \int_\Omega E(t-s,x,z)E(s,z,y)dz ds, \,\, x,y \in \Omega, \,\, t > 0.
\end{equation}
It is clear from \eqref{eq:E_simplification} that 
\begin{equation*}
    E(t,z,y) \neq 0 \text{ only if } z \in \supp(\tilde{\chi}_1) \subset \Omega \cap S,
\end{equation*}
hence the integrand in \eqref{eq:E*E_def} is well defined for all $z \in \Omega$. Now, for the rest of the proof, we assume $y \in \Omega_0$. If $\chi_1(x) = 0$, then by \eqref{eq:E_simplification}
    \begin{equation*}
        E(t-s,x,z)E(s,z,y) = [\Delta,\Tilde{\chi}_1(x)]H_S(t-s,x,z)\chi_1(z)[\Delta, \tilde{\chi}_1(z)]H_S(s,z,y).
    \end{equation*}
    This is zero for all $z \in \Omega$ since either $\chi_1(z) = 0$ or $[\Delta, \tilde{\chi}_1(z)] = 0$, hence $E*E(t,x,y) = 0$. Similarly, if $0 < \chi_1(x) < 1$, then $E(t-s,x,z) = 0$ for all $z \in \Omega$, making $E*E(t,x,y) = 0$. This shows that $E*E(t,x,y) \neq 0$ only if $\chi_1(x) = 1$. For such $x$, we have
    \begin{equation*}
        E(t-s,x,z)E(s,z,y) = [\Delta,\Tilde{\chi}_2(x)]H_\Omega(t-s,x,z)\chi_2(z)[\Delta, \tilde{\chi}_1(z)]H_S(s,z,y).
    \end{equation*}
    This is non-zero only if $0 < \tilde{\chi}_1(z) < 1$. For such $z$ and $y \in \Omega_0$, we have $|y-z| \geq \epsilon + 2\epsilon'$, hence
    \begin{equation} \label{eq:E_fresh_bound_with_epsilon}
        |E(s,z,y)| \leq Ce^{-\epsilon^2/(cs)}, \,\, 0 < s < t, \text{ for some } C > 0,
    \end{equation}
    by the same argument as with \eqref{eq:E_fresh_bound}. Combining this with the inequality $\frac{A}{t-s} + \frac{B}{s} \geq \frac{(\sqrt{A}+\sqrt{B})^2}{t}$ for $A,B > 0$ and $0 < s < t$, we obtain
    \begin{equation} \label{eq:E*E_fresh}
        \begin{split}
            |E*E(t,x,y)| &\leq \int_0^t \int_\Omega C e^{-(\epsilon')^2/(c(t-s))}Ce^{-\epsilon^2/(cs)} dzds \\
            &\leq C^2 e^{-(\epsilon+\epsilon')^2/(ct)} \int_0^t \int_\Omega dzds \\
            &= C^2|\Omega|t e^{-(\epsilon+\epsilon')^2/(ct)}, \,\, x \in \Omega, \,\, y \in \Omega_0, \,\, 0 < t < 1.
        \end{split}
    \end{equation}
    More generally, if we for $k \geq 2$ let $E^{*k}$ denote the iterated convolution of $E$ with itself $k$ times, then we have by induction
    \begin{equation} \label{eq:E^*_fresh}
        |E^{*k}(t,x,y)| \leq C^k|\Omega|^{k-1}\frac{t^{k-1}}{(k-1)!} e^{-(\epsilon+(k-1)\epsilon')^2/(ct)}, \,\, x \in \Omega, \,\, y \in \Omega_0, \,\, 0 < t < 1, \,\, k \geq 2.
    \end{equation}
    Indeed, by \eqref{eq:E*E_fresh} it holds for $k = 2$, and assuming it holds for some $k \geq 2$, we get
    \begin{align*}
        |E^{*(k+1)}(t,x,y)| &\leq \int_0^t \int_\Omega |E(t-s,x,z)||E^{*k}(s,z,y)|dzds \\
        &\leq \int_0^t \int_\Omega C e^{-(\epsilon')^2/(c(t-s))} C^k|\Omega|^{k-1}\frac{s^{k-1}}{(k-1)!} e^{-(\epsilon+(k-1)\epsilon')^2/(cs)} dzds \\
        &= C^{k+1}|\Omega|^{k-1}\frac{1}{(k-1)!} \int_0^t \int_\Omega s^{k-1} e^{-((\epsilon')^2/(c(t-s))+(\epsilon+(k-1)\epsilon')^2/(cs))} dzds \\
        &\leq C^{k+1}|\Omega|^{k-1}\frac{1}{(k-1)!} e^{-(\epsilon+k\epsilon')^2/(ct)}\int_0^t \int_\Omega s^{k-1} dzds \\
        &= C^{k+1}|\Omega|^{k}\frac{t^{k}}{k!} e^{-(\epsilon+k\epsilon')^2/(ct)}.
    \end{align*}
    Now define the Neumann series
    \begin{equation*}
        K(t,x,y) := \sum_{k=1}^\infty (-1)^{k+1} E^{*k}(t,x,y).
    \end{equation*}
    Here, $E^{*1} := E$. We start by only including $k \geq 2$ in the series and using \eqref{eq:E^*_fresh} to estimate 
    \begin{equation} \label{eq:K_fresh}
        \begin{split}
            \bigg|\sum_{k=2}^\infty (-1)^{k+1}E^{*k}(t,x,y)\bigg| &\leq \sum_{k=2}^\infty |E^{*k}(t,x,y)| \\
            &\leq \sum_{k=2}^\infty C^k|\Omega|^{k-1}\frac{t^{k-1}}{(k-1)!} e^{-(\epsilon+(k-1)\epsilon')^2/(ct)} \\
            &\leq Ce^{-\epsilon^2/(ct)} \sum_{k = 0}^\infty \frac{(C|\Omega|te^{-2\epsilon\epsilon'/(ct)})^k}{k!} \\
            &= Ce^{-\epsilon^2/(ct)}e^{C|\Omega|te^{-2\epsilon\epsilon'/(ct)}} \\
            &\leq Ce^{C|\Omega|}e^{-\epsilon^2/(ct)}, \,\, 0 < t < 1.
        \end{split}
    \end{equation}
    Finally, by \cite[Prop. 7]{sher2015conic}, 
    \begin{equation} \label{eq:Sher_prop}
        \begin{split}
            H_\Omega(t,x,y) - G(t,x,y) &= -G*K(t,x,y) \\
            &=-\int_0^t \int_\Omega G(t-s,x,z)K(s,z,y) dzds, \,\, x,y \in \Omega, \,\, t > 0.
        \end{split}
    \end{equation}
    Note that for $x,y \in \Omega_0$, the term in the integrand corresponding to $k = 1$ in the Neumann series becomes
    \begin{equation*}
        G(t-s,x,z)E(s,z,y) = \tilde{\chi}_2(x)H_\Omega(t-s,x,z)\chi_2(z)[\Delta, \tilde{\chi}_1(z)]H_S(s,z,y),
    \end{equation*}
    which is non-zero only if $0 < \tilde{\chi}_1(z) < 1$, hence
    \begin{align*}
        |x-z| \geq \epsilon + 2\epsilon', \\
        |y-z| \geq \epsilon + 2\epsilon'.
    \end{align*}
    Thus, we can argue as with \eqref{eq:E_fresh_bound} to obtain
    \begin{equation} \label{eq:GE_bound}
        \begin{split}
            |G(t-s,x,z)E(s,z,y)| &\leq C'e^{-\epsilon^2/(c(t-s))}e^{-\epsilon^2/(cs)} \\
            &\leq C'e^{-\epsilon^2/(ct)}, \,\, x,y \in \Omega_0, \,\, z \in \Omega, \,\, 0 < t < 1,
        \end{split}
    \end{equation}
    for some $C' > 0$. Combining this with \eqref{eq:K_fresh} gives for $x,y \in \Omega_0$ and $0 < t < 1$
    \begin{equation} \label{eq:G*K_fresh}
        \begin{split}
            |G*K(t,x,y)| &\leq \int_0^t \int_\Omega |G(t-s,x,z)||K(s,z,y)| dzds \\
            &\leq C'|\Omega|e^{-\epsilon^2/(ct)} + \int_0^t \int_\Omega (H_S(t-s,x,z)\chi_1(z) + \tilde{\chi}_2(x)H_\Omega(t-s,x,z)\chi_2(z)) Ce^{C|\Omega|}e^{-\epsilon^2/(cs)} dzds \\
            &\leq C'|\Omega|e^{-\epsilon^2/(ct)} + Ce^{C|\Omega|}e^{-\epsilon^2/(ct)} \bigg(\int_0^t \int_S H_S(s,x,z)\chi_1(z)dzds + \int_0^t \int_\Omega H_\Omega(s,x,z)\chi_2(z)dzds\bigg).
        \end{split}
    \end{equation}
    In the last step, we used that $\supp(\chi_1) \subset \Omega \cap S$ to change the region of integration from $\Omega$ to $S$ in the first term. Since
    \begin{equation*}
        \int_{D_j} H_{D_j}(s,x,z)\chi_j(z)dz
    \end{equation*}
    is the solution of the heat equation on $D_j$ with initial data $\chi_j(x)$, it approaches $\chi_j(x)$ as $s \downarrow 0$, and is therefore uniformly bounded for $0 < s < t$. Combining this with \eqref{eq:Sher_prop} and \eqref{eq:G*K_fresh} yields \eqref{eq:locality_principle_general}, completing the proof.
\end{proof}
\begin{remark} \label{remark:worse_exponent}
    In the literature (see e.g. \cite[Thm. 3.2.9, Thm. 5.5.6]{davies1989heat}), it is common to see heat kernel bounds of the following form: for any $\sigma > 0$, there is a $C > 0$ such that 
    \begin{equation*}
        |H_{D}(t,x,y)| \leq \frac{C}{t^{m}}e^{-|x-y|^2/((c+\sigma)t)}, \,\, x,y \in D, \,\, 0 < t < 1, 
    \end{equation*}
    for some $c,m > 0$, and similarly for the gradient. In particular, the denominator in the exponent is not $c$, but can be chosen arbitrarily close to $c$. The proof of Lemma \ref{lemma:locality_principle_general} shows that \eqref{eq:locality_principle_general} still holds in that case, without any $\sigma$. Indeed, for such bounds on the heat kernel, we would obtain
    \begin{equation*}
        \begin{split}
            |[\Delta, \Tilde{\chi}_j(x)]H_{D_j}(t,x,y)\chi_j(y)| &\leq \frac{C}{t^{m}}  e^{-(2\epsilon')^2/((c+\sigma) t)} \\
            &= \frac{C}{t^{m}} e^{-(3c-\sigma)(\epsilon')^2/(c(c+\sigma)t)}e^{-(\epsilon')^2/(c t)}.
        \end{split}
    \end{equation*}
    As long as $\sigma < 3c$,
    \begin{equation*}
        \frac{C}{t^{m}} e^{-(3c-\sigma)(\epsilon')^2/(c(c+\sigma)t)}
    \end{equation*}
    is bounded for $0 < t < 1$, hence \eqref{eq:E_fresh_bound} still holds. Similarly, if $|y-z| \geq \epsilon + 2\epsilon'$,
    then
    \begin{align*}
        |E(s,z,y)| &\leq \frac{C}{s^m} e^{-(\epsilon + 2\epsilon')^2/((c+\sigma)s)} \\
        &= \frac{C}{s^m}e^{-(4c\epsilon\epsilon' + 4c(\epsilon')^2 - \sigma \epsilon^2)/(c(c+\sigma)s)}e^{-\epsilon^2/(cs)}.
    \end{align*}
    Since 
    \begin{equation*}
        \frac{C}{s^m}e^{-(4c\epsilon\epsilon' + 4c(\epsilon')^2 - \sigma \epsilon^2)/(c(c+\sigma)s)}
    \end{equation*}
    is bounded for small enough $\sigma > 0$, \eqref{eq:E_fresh_bound_with_epsilon} and \eqref{eq:GE_bound} hold. Thus, the proof of Lemma \ref{lemma:locality_principle_general} works out as before. It is also worth emphasizing that we do not assume that $\Omega \subset S$. 
\end{remark}

\subsection{Partitioning of the polygon and locality principles}
Throughout this subsection, we assume that the Laplace operator on every domain is equipped with the Neumann boundary condition. To apply Lemma \ref{lemma:locality_principle_general} and derive locality estimates, we partition the polygon $\Omega$ into regions that are locally modeled by simpler domains: the infinite sector $S_\gamma$, the Euclidean half-plane $\R^2_+$, and the Euclidean plane $\R^2$. This partitioning in inspired by the construction in \cite{van1988heat}. 

Let $\gamma, R$, and the corner neighborhoods $B_i(R)$ be defined as in \eqref{eq:def_gamma}, \eqref{eq:def_R}, and \eqref{eq:def_B_i}, respectively. Then define
\begin{align*}
    C(R,\gamma) &:= \{x \in \Omega : d(x,\partial\Omega) < 2R\sin(\gamma/2)/3
    , \,\, x \notin \cup_{i=1}^n B_i(R)\}, \\
    D(R,\gamma) &:= \{x \in \Omega : x \notin \cup_{i=1}^n B_i(R), \,\, x \notin C(R,\gamma)\}.
\end{align*}
The set $C(R,\gamma)$ contains points near the edges but away from the corners, and $D(R,\gamma)$ contains points in the interior away from the boundary. The reason for requiring $x \in C(R,\gamma)$ to have $d(x,\partial \Omega) < 2R\sin(\gamma/2)/3$ instead of, say, $d(x,\partial \Omega) < R$, is to guarantee that $C(R,\gamma)$ consists of $n$ connected components, one near each edge. See Figure \ref{fig:polygon_partition}. In \cite{van1988heat}, they choose $d(x,\partial \Omega) < R\sin(\gamma/2)/2$, but we make a slight modification to make sure we get the same estimate for the heat kernels in the Neumann case as they get in the Dirichlet case.

Together, the sets form a partition of the polygon:
\begin{equation*}
    \Omega = \bigsqcup_{i=1}^n B_i(R) \sqcup C(R,\gamma) \sqcup D(R,\gamma).
\end{equation*}
Each of the regions is an exact geometric match with one of the model domains, allowing us to apply Lemma \ref{lemma:locality_principle_general} in each case. 
\begin{figure}[hbt!] 
\centering
\begin{tikzpicture}[scale=1]
  \def\radius{6}
  \def\arcradius{1.5} 

  \coordinate (O) at (0,0);

  \foreach \i in {1,...,6} {
    \pgfmathtruncatemacro{\j}{\i-1}
    \coordinate (P\i) at ({\radius*cos(60*\j)}, {\radius*sin(60*\j)});
  }

  \draw[thick] (P1) -- (P2) -- (P3) -- (P4) -- (P5) -- (P6) -- cycle;

  \foreach \i in {1,...,6} {
    \fill (P\i) circle[radius=2pt];
    \path (O) -- (P\i) coordinate[pos=1.08] (Label\i);
    \node at (Label\i) {$P_{\i}$};
  }

  \foreach \i/\angle in {1/120, 2/180, 3/240, 4/300, 5/0, 6/60} {
    \path (P\i) ++(\angle:\arcradius) coordinate (startArc\i);
    \draw[thick] (startArc\i) arc[start angle=\angle, end angle={\angle+120}, radius=\arcradius];
  }



    \draw[thick] (-1.882,4.196)--(1.882,4.196);

    \draw[thick] (-1.882,-4.196)--(1.882,-4.196);

    \draw[thick] (-4.575,0.468)--(-2.693,3.728);

    \draw[thick] (4.575,0.468)--(2.693,3.728);

    \draw[thick] (4.575,-0.468)--(2.693,-3.728);

    \draw[thick] (-4.575,-0.468)--(-2.693,-3.728);

    \node at (5.2,0) {$B_1(R)$};
    \node at (2.6,4.6) {$B_2(R)$};
    \node at (-2.6,4.6) {$B_3(R)$};
    \node at (-5.2,0) {$B_4(R)$};
    \node at (-2.6,-4.6) {$B_5(R)$};
    \node at (2.6,-4.6) {$B_6(R)$};

    \node at (0,-4.7) {$C(R,\gamma)$};
    \node at (0,4.7) {$C(R,\gamma)$};
    \node[rotate=60] at (4,-2.5) {$C(R,\gamma)$};
    \node[rotate=60] at (-4,2.5) {$C(R,\gamma)$};
    \node[rotate=-60] at (-4,-2.5) {$C(R,\gamma)$};
    \node[rotate=-60] at (4,2.5) {$C(R,\gamma)$};
    
    \node at (0,0) {\scalebox{1.5}{$D(R,\gamma)$}};
    
\end{tikzpicture}
\caption{Illustration of the partition of a polygon $\Omega$ into subsets used in the locality principle: the corner neighborhoods $B_i(R)$ (circular sectors around vertices $P_i$), the edge neighborhoods $C(R,\gamma)$ (regions close to edges but away from corners), and the interior region $D(R,\gamma)$ (points away from the boundary).}
\label{fig:polygon_partition}
\end{figure}
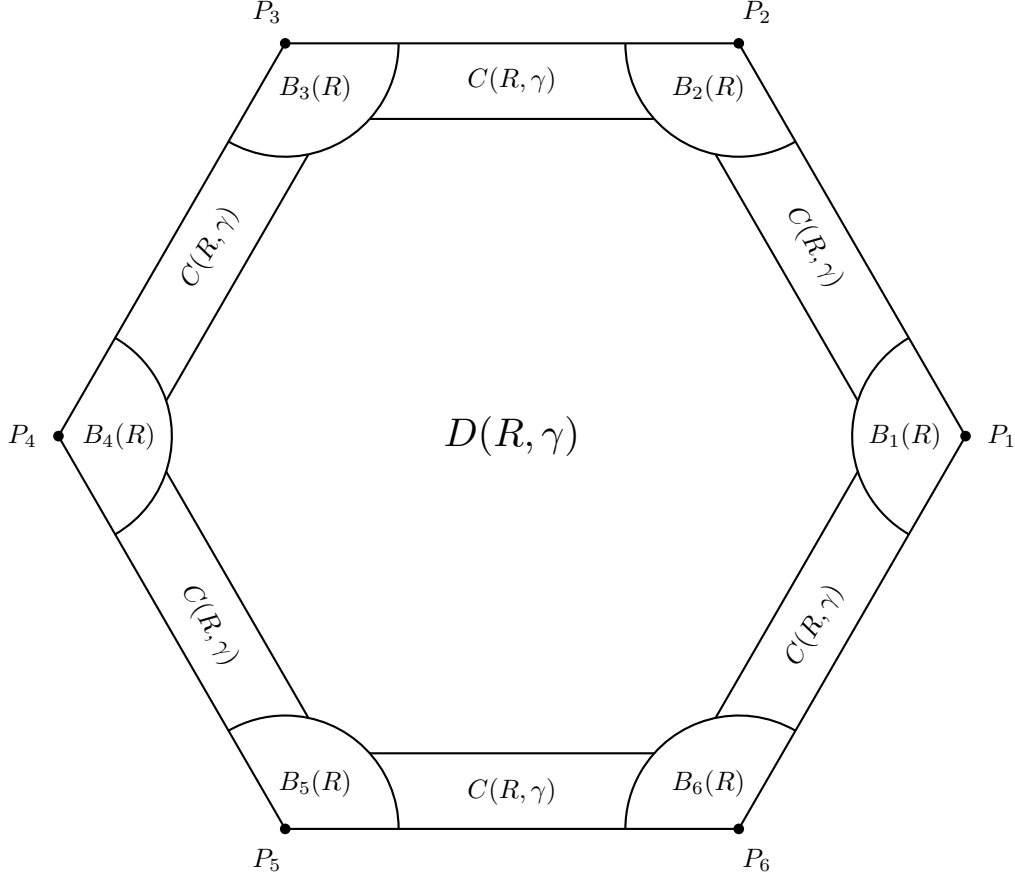
\begin{lemma}[Locality principle at a corner] \label{lemma:locality_principle_corner}
    Let $P_i$ be a corner of the polygon $\Omega$ with angle $\gamma_i$. 
    Then there is a $C > 0$ such that
    \begin{equation} \label{eq:locality_principle_corner}
    |H_\Omega(t,x,x) - H_{\gamma_i}(t,x,x)| \leq Ce^{-R^2\sin^2(\gamma/2)/(16t)}, \,\, x \in B_i(R), \,\, 0 < t < 1.
\end{equation}
\end{lemma}
\begin{proof}
    We need to verify that $\Omega$ and $S_{\gamma_i}$ satisfy the bounds \eqref{eq:H_S_bounds}. By \cite[Thm. 2.4.4]{davies1989heat}, we have the diagonal estimates
    \begin{equation*}
        |H_{D_j}(t,x,x)| \leq \frac{C}{t}, \,\, x \in D_j, \,\, 0 < t < 1,
    \end{equation*}
    for some $C > 0$. Here, $D_1 = \Omega$ and $D_2 = S_{\gamma_i}$. Moreover, as noted in \cite[p. 252]{nursultanov2019hear}, we may apply \cite[Thm. 1.1]{yan2010gradient} to $D_j$ for short time, which gives   
    \begin{equation*}
        |\nabla H_{D_j}(t,x,x)| \leq \frac{C}{t^{3/2}}, \,\, x \in D_j, \,\, 0 < t < 1.
    \end{equation*}
    Then, as explained in e.g. \cite[Thm. 1.1]{yan2010gradient} and \cite[Thm. 4.1, Thm. 4.9]{coulhon2006gaussian}, off-diagonal estimates of the form
    \begin{align*}
        |H_{D_j}(t,x,y)| \leq \frac{C}{t^m}e^{-|x-y|^2/(4t)}, \,\, x,y \in \Omega \cap D_j, \,\, 0 < t < 1, \\
        |\nabla H_{D_j}(t,x,y)| \leq \frac{C}{t^m}e^{-|x-y|^2/(4t)}, \,\, x,y \in \Omega \cap D_j, \,\, 0 < t < 1, 
    \end{align*}
    follow from the self-improvement property, combined with boundedness of $\Omega$. Now let $\epsilon,\delta > 0$ be such that $\Omega$ and $S_{\gamma_i}$ are exact geometric matches on $B_i(R+\epsilon+\delta)$. Then Lemma \ref{lemma:locality_principle_general} gives
    \begin{equation} \label{eq:locality_principle_corner_epsilon}
        |H_\Omega(t,x,x) - H_{\gamma_i}(t,x,x)| \leq Ce^{-\epsilon^2/(4t)}, \,\, x \in B_i(R), \,\, 0 < t < 1,
    \end{equation}
    for some $C > 0$. By definition of $R$, we may let $\epsilon + \delta = R$. Choosing 
    \begin{align*}
        \epsilon &= R\sin(\gamma/2)/2, \\
        \delta &= R - R\sin(\gamma/2)/2
    \end{align*}
    gives the desired bound.
    
\end{proof}

\begin{lemma}[Locality principle at an edge] \label{lemma:locality_principle_edge}
    Let $\Omega$ be a polygon. There is a $C > 0$ such that
    \begin{equation} \label{eq:locality_principle_edge}
    |H_\Omega(t,x,x) - H_{\R^2_+}(t,x,x)| \leq Ce^{-R^2\sin^2(\gamma/2)/(16t)}, \,\, x \in C(R,\gamma), \,\, 0 < t < 1. 
    \end{equation}
\end{lemma}
\begin{proof}
    Fix an edge $e$ of $\Omega$. Choose coordinates so that $x = 0$ at the midpoint of $e$, $y = 0$ on $e$, and $y$ increases as you move from $e$ towards the interior of the polygon. Then, by \eqref{eq:heat_kernel_R^2_+},
    \begin{align*}
        H_{\R^2_+}(t,x,y) &= \frac{e^{-|x-y|^2/(4t)}+e^{-|x-y^*|^2/(4t)}}{4\pi t}, \\
        |\nabla H_{\R^2_+}(t,x,y)| &\leq \frac{C\sqrt{|x|^2+|y|^2}}{t}H_{\R^2_+}(t,x,y),
    \end{align*}
    for some $C > 0$. Here, $y^* = (y_1,-y_2)$ is the reflection of $y = (y_1,y_2)$ across the edge of the polygon. Since 
    \begin{equation*}
        |x-y| \leq |x-y^*|, \,\, x,y \in \Omega \cap \R^2_+,
    \end{equation*}
    and $\Omega$ is bounded, there are $C_1,C_2 > 0$ such that
    \begin{align*}
        |H_{\R^2_+}(t,x,y)| &\leq \frac{C_1}{t} e^{-|x-y|^2/(4t)}, \,\, x,y \in \Omega \cap \R^2_+, \,\, 0< t < 1, \\
        |\nabla H_{\R^2_+}(t,x,y)| &\leq \frac{C_2}{t^2} e^{-|x-y|^2/(4t)}, \,\, x,y \in \Omega \cap \R^2_+, \,\, 0 < t < 1.
    \end{align*}
    As explained in the proof of Lemma \ref{lemma:locality_principle_corner}, analogous bounds hold for $H_\Omega(t,x,y)$. Thus, for $\epsilon, \delta > 0$ such that $\Omega$ and $\R^2_+$ are exact geometric matches on an $(\epsilon+\delta)$-neighborhood of $C(R,\gamma)$, Lemma \ref{lemma:locality_principle_general} gives that there is a $C > 0$ such that
    \begin{equation*}
        |H_\Omega(t,x,x) - H_{\R^2_+}(t,x,x)| \leq Ce^{-\epsilon^2/(4t)}, \,\, x \in C(R,\gamma), \,\, 0 < t < 1.
    \end{equation*}
    Since $d(C(R,\gamma),\partial\Omega \backslash e) > R$, we may let $\epsilon + \delta = R$.
    Choosing 
    \begin{align*}
        \epsilon &= R\sin(\gamma/2)/2, \\
        \delta &= R - R\sin(\gamma/2)/2
    \end{align*}
    gives the desired bound.
\end{proof}

\begin{lemma}[Locality principle at the interior] \label{lemma:locality_principle_interior}
    Let $\Omega$ be a polygon. There is a $C > 0$ such that
    \begin{equation} \label{eq:locality_principle_interior}
    |H_\Omega(t,x,x) - H_{\R^2}(t,x,x)| \leq Ce^{-R^2\sin^2(\gamma/2)/(16t)}, \,\, x \in D(R,\gamma), \,\, 0 < t < 1.
\end{equation}
\end{lemma}
\begin{proof}
    By \eqref{eq:heat_kernel_R^2}, we have
    \begin{align*}
        H_{\R^2}(t,x,y) &= \frac{1}{4\pi t}e^{-|x-y|^2/(4t)}, \\
        \nabla H_{\R^2}(t,x,y) &= -\frac{x-y}{2t}H_{\R^2}(t,x,y).
    \end{align*}
    In particular, there are $C_1, C_2 > 0$ such that
    \begin{align*}
        |H_{\R^2}(t,x,y)| &= \frac{C_1}{t}e^{-|x-y|^2/(4t)}, \,\, x,y \in \Omega, \,\, 0 < t < 1,\\
        |\nabla H_{\R^2}(t,x,y)| &= \frac{C_2}{t^2}e^{-|x-y|^2/(4t)}, \,\, x,y \in \Omega, \,\, 0 < t < 1.
    \end{align*}
    Since we have analogous bounds for $H_\Omega(t,x,y)$ by the proof of Lemma \ref{lemma:locality_principle_corner}, Lemma \ref{lemma:locality_principle_general} gives that there is a $C > 0$ such that
    \begin{equation*}
        |H_\Omega(t,x,x) - H_{\R^2}(t,x,x)| \leq Ce^{-\epsilon^2/(4t)}, \,\, x \in D(R,\gamma), \,\, 0 < t < 1,
    \end{equation*}
    where $\epsilon,\delta > 0$ are such that $\Omega$ and $\R^2$ are exact geometric matches on an $(\epsilon+\delta)$-neighborhood of $D(R,\gamma)$. In particular, $\epsilon$ must be chosen so that an $\epsilon$-neighborhood of $D(R,\gamma)$ has a positive distance to $\partial\Omega$. 
    Since $d(D(R),\partial \Omega) \geq 2R\sin(\gamma/2)/3$, we may let 
    \begin{align*}
        \epsilon &= R\sin(\gamma/2)/2,
    \end{align*}
    which gives the desired bound.
\end{proof}
\section{Heat trace} \label{sec:heat_trace}
Having established locality estimates for the Neumann heat kernel $H_\Omega(t,x,x)$ on different regions of the polygon, we now turn to the proof of Theorem \ref{thm:main_result}. The key idea is to approximate $H_\Omega(t,x,x)$ by the corresponding model heat kernels on each subregion of $\Omega$, using Lemmas \ref{lemma:locality_principle_corner}-\ref{lemma:locality_principle_interior}, and then integrate the model heat kernels over each subregion. Applying these lemmas region by region yields
\begin{align} \label{eq:heat_trace_estimate}
    \begin{split}
        &\left|Z(t) - \left(\sum_{i=1}^n \int_{B_i(R)} H_{\gamma_i}(t,x,x) dx + \int_{C(R,\gamma)} H_{\R^2_+}(t,x,x) dx + \int_{D(R,\gamma)} H_{\R^2}(t,x,x) dx\right)\right| \\
        &\leq \sum_{i=1}^n \int_{B_i(R)} |H_\Omega(t,x,x)-H_{\gamma_i}(t,x,x)|dx + \int_{C(R,\gamma)} |H_\Omega(t,x,x)-H_{\R^2_+}(t,x,x)|dx \\ 
        &+ \int_{D(R,\gamma)} |H_\Omega(t,x,x)-H_{\R^2}(t,x,x)|dx \\
        &\leq Ce^{-R^2\sin^2(\gamma/2)/(16t)}\sum_{i=1}^n|B_i(R)| + Ce^{-R^2\sin^2(\gamma/2)/(16t)}|C(R,\gamma)| + Ce^{-R^2\sin^2(\gamma/2)/(16t)}|D(R,\gamma)| \\
        &= C|\Omega|e^{-R^2\sin^2(\gamma/2)/(16t)}, \,\, 0 < t < 1.
    \end{split}
\end{align}
Next, we compute the contributions from each model heat kernel explicitly, following the approach of \cite{van1988heat}.

\subsection{Heat trace contribution from a corner}
In polar coordinates, we have 
\begin{equation*}
    H_\gamma(t,r,\theta,r',\theta') = \mathcal{L}^{-1}(G_\gamma(s,r,\theta,r',\theta'))(t),
\end{equation*}
where $\mathcal{L}^{-1}$ is the inverse Laplace transform and $G_\gamma(s,r,\theta,r',\theta')$ is the Green's function for the infinite sector. By \cite[Eq. (2.3)]{nursultanov2024heat}, it is given by
\begin{equation*}
    G_\gamma(s,r,\theta,r',\theta') = \frac{A+B+C}{\pi^2}
\end{equation*}
with 
\begin{align*}
    A &= \int_0^\infty K_{ix}(r\sqrt{s})K_{ix}(r'\sqrt{s})\cosh((\pi-|\theta'-\theta|)x) dx,\\
    B &= \int_0^\infty K_{ix}(r\sqrt{s})K_{ix}(r'\sqrt{s})\frac{\sinh(\pi x)}{\sinh(\gamma x)}\cosh((\theta+\theta'-\gamma)x) dx,\\
    C &= \int_0^\infty K_{ix}(r\sqrt{s})K_{ix}(r'\sqrt{s})\frac{\sinh((\pi-\gamma)x)}{\sinh(\gamma x)}\cosh((\theta-\theta')x) dx.
\end{align*}
Here, $K_\nu$ are the modified Bessel functions of the second kind. For the Dirichlet boundary condition, the Green's function is instead given by $G_\gamma(s,r,\theta,r',\theta') = \frac{A-B+C}{\pi^2}$. 
In \cite[Thm. 2]{van1988heat}, they compute
\begin{align*}
    \frac{1}{\pi^2}\int_0^{\gamma} \int_0^R \mathcal{L}^{-1}(A)rdrd\theta &= \frac{\gamma R^2}{8\pi t}, \\
     \frac{1}{\pi^2}\int_0^{\gamma} \int_0^R \mathcal{L}^{-1}(B)rdrd\theta &= \frac{R^2}{2\pi t} \int_0^1 \frac{\sqrt{1-y^2}}{e^{R^2y^2/t}} dy ,\\
      \frac{1}{\pi^2}\int_0^{\gamma} \int_0^R \mathcal{L}^{-1}(C)rdrd\theta &= \frac{\pi^2 - \gamma^2}{24\pi\gamma} + A_\gamma(t),
\end{align*}
where 
\begin{align*}
    A_\gamma(t) =
    \begin{dcases}
        \frac{1}{4\pi} \sin\!\biggl(\frac{\pi^2}{\gamma}\biggr) 
        \int_0^\infty \frac{e^{-R^2(1+\cosh(q))/(2t)}}{(1+\cosh(q))(\cosh(\pi q/\gamma) - \cos(\pi^2/\gamma))} \, dq, 
        & \frac{\pi}{2} < \gamma \leq 2\pi, \\[2ex]
        -\frac{\gamma}{2\pi} \sum_{n=1}^{M-1} 
        \frac{e^{-R^2(1-\cos(2n\gamma))/(2t)}}{1 - \cos(2n\gamma)} 
        - \frac{\gamma}{8\pi}e^{-R^2/t}, 
        & \gamma = \frac{\pi}{2M}, \, M = 1,2,\dots, \\[2ex]
        -\frac{\gamma}{2\pi} \sum_{n=1}^{M} 
        \frac{e^{-R^2(1-\cos(2n\gamma))/(2t)}}{1 - \cos(2n\gamma)} \\
        + \frac{1}{4\pi} \sin\!\biggl(\frac{\pi^2}{\gamma}\biggr) 
        \int_0^\infty \frac{e^{-R^2(1+\cosh(q))/(2t)}}{(1+\cosh(q))(\cosh(\pi q/\gamma) - \cos(\pi^2/\gamma))} \, dq,
        & \frac{\pi}{2M+2} < \gamma < \frac{\pi}{2M}, \, M = 1,2,\dots.
    \end{dcases}
\end{align*}
In particular, we have by \cite[Cor. 3]{van1988heat}
\begin{equation} \label{eq:A(t)}
    |A_\gamma(t)| \leq \bigg(\frac{\gamma}{8\pi} + \frac{3\pi^2}{64\gamma^2}\bigg) e^{-R^2\sin^2(\gamma)/t}, \,\, t > 0.
\end{equation}
Thus,
\begin{align} \label{eq:heat_trace_contribution_corner}
    \begin{split}
        \int_{B_i(R)} H_{\gamma_i}(t,r,\theta,r,\theta)rdrd\theta &= \frac{1}{\pi^2}\int_0^{\gamma} \int_0^R (\mathcal{L}^{-1}(A) + \mathcal{L}^{-1}(B) + \mathcal{L}^{-1}(C))rdrd\theta \\
        &= \frac{\gamma_i R^2}{8\pi t} + \frac{R^2}{2\pi t} \int_0^1 \frac{\sqrt{1-y^2}}{e^{R^2y^2/t}} dy + \frac{\pi^2 - \gamma_i^2}{24\pi\gamma_i} + A_{\gamma_i}(t).
    \end{split}
\end{align}

\subsection{Heat trace contribution from the edges} \label{S:contribution_to_edges_Neumann}
From \eqref{eq:heat_kernel_R^2_+}, it follows that
\begin{align*}
    \int_{C(R,\gamma)} H_{\R_+^2}(t,x,x)dx &= \int_{C(R,\gamma)} \frac{1+e^{-d(y,\partial\Omega)^2/t}}{4\pi t} dy \\
    &= \frac{|C(R,\gamma)|}{4\pi t} + \frac{1}{4\pi t}\int_{C(R,\gamma)} e^{-d(y,\partial\Omega)^2/t} dy.
\end{align*}
Since $C(R,\gamma)$ consists of $n$ connected components, one near each edge of $\Omega$, we have
\begin{equation*}
    \int_{C(R,\gamma)} e^{-d(y,\partial\Omega)^2/t} dy = \int_0^{2R\sin(\gamma/2)/3} e^{-y^2/t} L(y) dy,
\end{equation*}
where $L(y)$ is the sum of the lengths of the line segments with distance $y$ to $\partial\Omega$ contained in $C(R,\gamma)$. It is a geometric exercise to see that
\begin{equation*}
    L(y) = |\partial\Omega| - 2n\sqrt{R^2 - y^2}, \,\, 0 \leq y \leq \frac{2R\sin(\gamma/2)}{3}.
\end{equation*}
Hence
\begin{align*}
    \int_{C(R,\gamma)} e^{-d(y,\partial\Omega)^2/t} dy &= |\partial\Omega|\int_0^{2R\sin(\gamma/2)/3} e^{-y^2/t} dy - 2n \int_0^{2R\sin(\gamma/2)/3} \frac{\sqrt{R^2 - y^2}}{e^{y^2/t}} dy \\
    &= \frac{|\partial\Omega|\sqrt{\pi t}}{2} - |\partial\Omega|\int_{2R\sin(\gamma/2)/3}^\infty e^{-y^2/t}dy - 2nR^2\int_0^{2\sin(\gamma/2)/3} \frac{\sqrt{1-y^2}}{e^{R^2y^2/t}} dy.
\end{align*}
Thus,
\begin{align} \label{eq:heat_trace_contribution_edges}
    \begin{split}
        &\int_{C(R,\gamma)} H_{\R_+^2}(t,x,x)dx \\
        &= \frac{|C(R,\gamma)|}{4\pi t} + \frac{1}{4\pi t} \bigg(\frac{|\partial\Omega|\sqrt{\pi t}}{2} - |\partial\Omega|\int_{2R\sin(\gamma/2)/3}^\infty e^{-y^2/t}dy - 2nR^2\int_0^{2\sin(\gamma/2)/3} \frac{\sqrt{1-y^2}}{e^{R^2y^2/t}} dy\bigg) \\
        &= \frac{|C(R,\gamma)|}{4\pi t} + \frac{|\partial\Omega|}{8\sqrt{\pi t}} - \frac{|\partial\Omega|}{4\pi t}\int_{2R\sin(\gamma/2)/3}^\infty e^{-y^2/t}dy - \frac{nR^2}{2\pi t}\int_0^{2\sin(\gamma/2)/3} \frac{\sqrt{1-y^2}}{e^{R^2y^2/t}} dy.
    \end{split}
\end{align}

\subsection{Heat trace contribution from the interior}
The contribution from the interior is simply
\begin{equation} \label{eq:heat_trace_contribution_interior}
    \int_{D(R,\gamma)} H_{\R^2}(t,x,x) dx = \int_{D(R,\gamma)} \frac{1}{4\pi t} dx = \frac{|D(R,\gamma)|}{4\pi t}.
\end{equation}

\subsection{Proof of Theorem \ref{thm:main_result}}
We now insert \eqref{eq:heat_trace_contribution_corner}, \eqref{eq:heat_trace_contribution_edges}, \eqref{eq:heat_trace_contribution_interior} into \eqref{eq:heat_trace_estimate} to obtain
\begin{align*}
    \bigg|Z(t) &- \sum_{i=1}^n\left[\frac{\gamma_i R^2}{8\pi t} + \frac{R^2}{2\pi t} \int_0^1 \frac{\sqrt{1-y^2}}{e^{R^2y^2/t}} dy + \frac{\pi^2 - \gamma_i^2}{24\pi\gamma_i} + A_{\gamma_i}(t)\right] \\
    &- \left[\frac{|C(R,\gamma)|}{4\pi t} + \frac{|\partial\Omega|}{8\sqrt{\pi t}} - \frac{|\partial\Omega|}{4\pi t}\int_{2R\sin(\gamma/2)/3}^\infty e^{-y^2/t}dy - \frac{nR^2}{2\pi t}\int_0^{2\sin(\gamma/2)/3} \frac{\sqrt{1-y^2}}{e^{R^2y^2/t}} dy\right] 
    - \frac{|D(R,\gamma)|}{4\pi t}\bigg| \\ 
    &\leq C|\Omega|e^{-R^2\sin^2(\gamma/2)/(16t)}, \,\, 0 < t < 1.
\end{align*}
Since
\begin{equation*}
    \sum_{i=1}^n \frac{\gamma_i R^2}{8\pi t} + \frac{|C(R,\gamma)|}{4\pi t} + \frac{|D(R,\gamma)|}{4\pi t} = \frac{|\Omega|}{4\pi t},
\end{equation*}
this simplifies to
\begin{align*}
    \bigg|Z(t) &- \frac{|\Omega|}{4\pi t} - \frac{|\partial\Omega|}{8\sqrt{\pi t}} - \sum_{i=1}^n \frac{\pi^2-\gamma_i^2}{24\pi\gamma_i} - \frac{nR^2}{2\pi t} \int_0^1 \frac{\sqrt{1-y^2}}{e^{R^2y^2/t}} dy - \sum_{i=1}^n  A_{\gamma_i}(t) \\
    &+ \frac{|\partial\Omega|}{4\pi t}\int_{2R\sin(\gamma/2)/3}^\infty e^{-y^2/t}dy + \frac{nR^2}{2\pi t}\int_0^{2\sin(\gamma/2)/3} \frac{\sqrt{1-y^2}}{e^{R^2y^2/t}} dy\bigg| 
     \leq C|\Omega|e^{-R^2\sin^2(\gamma/2)/(16t)}, \,\, 0 < t < 1.
\end{align*}
Since $\frac{2\sin(\gamma/2)}{3} < 1$, we simplify further into
\begin{align*}
    \bigg|Z(t) &- \frac{|\Omega|}{4\pi t} - \frac{|\partial\Omega|}{8\sqrt{\pi t}} - \sum_{i=1}^n \frac{\pi^2-\gamma_i^2}{24\pi\gamma_i} - \frac{nR^2}{2\pi t} \int_{2\sin(\gamma/2)/3}^1 \frac{\sqrt{1-y^2}}{e^{R^2y^2/t}} dy - \sum_{i=1}^n  A_{\gamma_i}(t) + \frac{|\partial\Omega|}{4\pi t}\int_{2R\sin(\gamma/2)/3}^\infty e^{-y^2/t}dy\bigg| \\
    & \leq C|\Omega|e^{-R^2\sin^2(\gamma/2)/(16t)}, \,\, 0 < t < 1.
\end{align*}
By the triangle inequality,
\begin{align*}
    &\bigg|Z(t) - \frac{|\Omega|}{4\pi t} - \frac{|\partial\Omega|}{8\sqrt{\pi t}} - \sum_{i=1}^n \frac{\pi^2-\gamma_i^2}{24\pi\gamma_i}\bigg| \\ 
    \leq &\frac{nR^2}{2\pi t} \int_{2\sin(\gamma/2)/3}^1 \frac{\sqrt{1-y^2}}{e^{R^2y^2/t}} dy + \sum_{i=1}^n |A_{\gamma_i}(t)| + \frac{|\partial\Omega|}{4\pi t}\int_{2R\sin(\gamma/2)/3}^\infty e^{-y^2/t}dy + C|\Omega|e^{-R^2\sin^2(\gamma/2)/(16t)}, \,\, 0 < t < 1.
\end{align*}
We estimate
\begin{align*}
    \int_{2\sin(\gamma/2)/3}^1 \frac{\sqrt{1-y^2}}{e^{R^2y^2/t}} dy &\leq \int_{2\sin(\gamma/2)/3}^1 e^{-R^2y^2/t} dy \leq e^{-4R^2\sin^2(\gamma/2)/(9t)}, \\ 
    \int_{2R\sin(\gamma/2)/3}^\infty e^{-y^2/t}dy &= 
    -\frac{t}{2y}e^{-y^2/t}\bigg|_{2R\sin(\gamma/2)/3}^\infty - \int_{2R\sin(\gamma/2)/3}^\infty \frac{t}{2y^2}e^{-y^2/t} dy \leq \frac{3t}{4R\sin(\gamma/2)}e^{-4R^2\sin^2(\gamma/2)/(9t)}.
\end{align*}
Moreover, \eqref{eq:A(t)} gives
\begin{align*}
    \sum_{i=1}^n |A_{\gamma_i}(t)| &\leq \sum_{i=1}^n \bigg(\frac{\gamma_i}{8\pi} + \frac{3\pi^2}{64\gamma_i^2}\bigg) e^{-R^2\sin^2(\gamma_i)/t} \\
    &\leq \frac{1}{8\pi}e^{-R^2\sin^2(\gamma)/t} \sum_{i=1}^n \gamma_i + \frac{3\pi^2}{64}e^{-R^2\sin^2(\gamma)/t} \sum_{i=1}^n \frac{1}{\gamma_i^2} \\
    &\leq \bigg(\frac{n-2}{8} + \frac{3\pi^2n}{64\gamma^2}\bigg)e^{-R^2\sin^2(\gamma)/t}.
\end{align*}
Hence, for $0 < t < 1$,
\begin{align*}
    \bigg|Z(t) - \frac{|\Omega|}{4\pi t} - \frac{|\partial\Omega|}{8\sqrt{\pi t}} - \sum_{i=1}^n \frac{\pi^2-\gamma_i^2}{24\pi\gamma_i}\bigg| &\leq \left(\frac{nR^2}{2\pi t} + \frac{3|\partial\Omega|}{16\pi R\sin(\gamma/2)}\right)e^{-4R^2\sin^2(\gamma/2)/(9t)} \\
    &+ \left(\frac{n-2}{8} + \frac{3\pi^2n}{64\gamma^2}\right) e^{-R^2\sin^2(\gamma)/t} 
    + C|\Omega|e^{-R^2\sin^2(\gamma/2)/(16t)} \\
    &\leq \left(\frac{nR^2}{2\pi t} + \frac{3|\partial\Omega|}{16\pi R\sin(\gamma/2)} + \frac{n-2}{8} + \frac{3\pi^2n}{64\gamma^2} + C|\Omega|\right)e^{-R^2\sin^2(\gamma/2)/(16t)}.
\end{align*}
In particular, there is a $C > 0$ such that
\begin{equation} \label{eq:main_result_small_time}
    \bigg|Z(t) - \frac{|\Omega|}{4\pi t} - \frac{|\partial\Omega|}{8\sqrt{\pi t}} - \sum_{i=1}^n \frac{\pi^2-\gamma_i^2}{24\pi\gamma_i}\bigg| \leq Ce^{-R^2\sin^2(\gamma/2)/(16t)}, \,\, 0 < t < 1.
\end{equation}
This proves Theorem \ref{thm:main_result} for $0 < t < 1$. For $t \geq 1$, note that $Z(t)$ is bounded. Indeed, by Weyl's law we have
\begin{equation*} 
    \lambda_k \sim \frac{4\pi k}{|\Omega|}, \,\, k \uparrow \infty,
\end{equation*}
which implies that there is a $C > 0$ such that
\begin{equation} \label{eq:Heat_trace_large_time}
    Z(t) \leq 1 + Ce^{-\lambda_1 t}, \,\, t \geq 1.
\end{equation}
Therefore,
    \begin{align*}
        \bigg|Z(t) - \frac{|\Omega|}{4\pi t} - \frac{|\partial\Omega|}{8\sqrt{\pi t}} - \sum_{i=1}^n \frac{\pi^2-\gamma_i^2}{24\pi\gamma_i}\bigg| \leq Z(t) + \frac{|\Omega|}{4\pi t} + \frac{|\partial\Omega|}{8\sqrt{\pi t}} + \sum_{i=1}^n \frac{\pi^2-\gamma_i^2}{24\pi\gamma_i}
    \end{align*}
    is also bounded for $t \geq 1$. 
    It follows that there is a $C > 0$ such that \eqref{eq:main_result_small_time} holds for $t \geq 1$, from which Theorem \ref{thm:main_result} follows. \qed

\section{Outlook} \label{sec:outlook}
In this section, we discuss several directions for future research. These include the problem of estimating the coefficient in front of the exponential error bound, finding the sharp exponent in the remainder term, and generalizing the results to Robin boundary conditions. Each of these directions raises its own set of difficulties, and we briefly discuss them below.

\subsection{Estimating the coefficient in front of the exponential}
In Theorem \ref{thm:main_result}, the remainder in the asymptotic expansion is of the form $Ce^{-\alpha/t}$. While we have obtained an explicit estimate for the exponent $\alpha$, the coefficient $C$ remains unspecified. It is therefore natural to ask whether one can obtain an upper bound for $C$. A central difficulty in this problem is to obtain control of the constants appearing in Gaussian estimates for the heat kernels and their gradients on the relevant domains. Indeed, the proofs of Lemmas \ref{lemma:locality_principle_general}-\ref{lemma:locality_principle_interior} rely on such estimates on model domains, including infinite sectors and half-planes. While Gaussian upper bounds are well known in these settings, extracting explicit constants appears to be non-trivial.
A further difficulty arises in extending the estimate to all $t > 0$, which requires an explicit bound on the heat trace $Z(t)$ for large $t$. While we have \eqref{eq:Heat_trace_large_time} by Weyl's law, obtaining explicit constants would require detailed information about the spectrum of the Neumann Laplacian on the polygon, which is typically not available in closed form. For the Dirichlet boundary condition, one could inscribe the polygon by a rectangle and apply the domain monotonicity property \cite[Thm. 6.20]{borthwick2020spectral}. However, such a domain monotonicity property does not hold in the Neumann case. In summary, while it is plausible that an explicit bound on $C$ could be obtained, this would require a systematic analysis of the constants in the underlying heat kernel estimates, which lies beyond the scope of the present work.

\subsection{Sharp exponent}
In this article, we obtained an explicit estimate for the exponential decay in the remainder term of the Neumann heat trace expansion of polygons. To the best of our knowledge, this is the first time such an estimate has been obtained. However, it is not believed to be sharp. In \cite{maardby2025spectral}, the precise exponential decay rate is determined for the so-called \textit{integrable polygons}, those whose interior angles are of the form $\pi/n$ with $n \in \N$. This class consists precisely of rectangles, equilateral triangles, isosceles right triangles, and hemi-equilateral triangles. In each of these cases, explicit formulas for the Dirichlet and Neumann spectrum are available, allowing for a detailed analysis of the heat trace asymptotics. The result shows that for any integrable polygon, the exponential decay rate in the remainder term of the Dirichlet and Neumann heat traces is given by $\frac{L^2}{4}$, where $L$ denotes the length of the shortest closed geodesic of the polygon. More precisely, the remainder decays like $\frac{1}{t}e^{-L^2/(4t)}$. This result suggests a natural conjecture: for any polygon, the remainder in the Dirichlet and Neumann heat traces decays like $\frac{1}{t}e^{-L^2/(4t)}$. 
However, the techniques developed in this article, based on local geometric approximations and heat kernel comparisons, do not appear sufficient to recover this sharp exponent in the general case. Achieving the sharp rate of decay likely requires a more global spectral analysis, capable of capturing the contribution of closed geodesics, as well as refined asymptotics for the eigenvalues. Developing such methods remains an open problem and would represent a significant advance in the understanding of spectral invariants of general polygonal domains.

\subsection{Robin boundary conditions}
It is natural to ask whether the methods developed in this article can be extended to treat the case of Robin boundary conditions. In \cite[Sec. 2.4]{nursultanov2019hear}, locality principles of the form \eqref{eq:weak_bound} are established. By following their argument, it is straightforward to verify that locality principles analogous to Lemma \ref{lemma:locality_principle_general} continue to hold for Robin heat kernels, once they are known in the Neumann case. However, the exponential decay is weaker in the Robin case than in the Neumann case. The main difficulty in extending Theorem \ref{thm:main_result} to Robin boundary conditions lies in computing the contribution of the model heat kernels to the heat trace. For the infinite sector $S_\gamma$, the Green's function for Robin boundary conditions can be obtained using the techniques in \cite[Appendix A]{nursultanov2024heat}, but computing its contribution to the heat trace appears to be a non-trivial task. Moreover, the heat kernel on the Euclidean half-plane with Robin boundary conditions contains an additional term compared to the Neumann heat kernel. This extra term contributes a non-exponentially decaying series to the heat trace. Specifically, it gives rise to an $\mathcal{O}(1)$ term, followed by terms of order $\sqrt{t}$, $t$, and so on; see \cite[Sec. 3.2]{nursultanov2019hear}. As a result, the heat trace expansion of polygons with Robin boundary conditions appears to contain infinitely many non-zero terms, in contrast to the three-term structure followed by exponential decay that holds in the Dirichlet and Neumann cases. To summarize, while the locality principles can be adapted to the Robin case, computing the full asymptotic expansion of the heat trace appears to be significantly more difficult.


\end{document}